\documentclass{article}
\usepackage{amsmath, amssymb, amsthm, amsfonts}

\newcommand{\ol}{\overline}
\newcommand{\pa}{\partial}

\newcommand{\de}{\delta}

\newcommand{\ka}{\kappa}

\newcommand{\ve}{\varepsilon}

\newcommand{\Om}{\Omega}
\newcommand{\cd}{\cdot}
\renewcommand{\th}{\theta}

\newcommand{\R}{{\mathbb R}}

\newcommand{\Z}{{\mathbb Z}}
\newcommand{\N}{{\mathbb N}}
\newcommand{\T}{{\mathbb T}}

\newcommand{\tE}{\widetilde{E}}

\renewcommand{\(}{\left(}
\renewcommand{\)}{\right)}

\newtheorem{theorem}{\bf Theorem}[section]
\newtheorem{lemma}[theorem]{\bf Lemma}
\newtheorem{proposition}[theorem]{\bf Proposition}

\theoremstyle{remark}
  \newtheorem{remark}[theorem]{\sc Remark}
\theoremstyle{definition}

\numberwithin{equation}{section}

\begin{document}

\title{Global solvability for semi-discrete Kirchhoff equation}
\author{
Fumihiko Hirosawa\footnote{Department of Mathematical Sciences, Faculty of Science, Yamaguchi University, Japan; 
e-mail: hirosawa@yamaguchi-u.ac.jp}
}
\date{}
\maketitle

\begin{abstract}
In this paper, we consider the global solvability and energy conservation for initial value problem of nonlinear semi-discrete wave equation of Kirchhoff type, which is a discretized model of Kirchhoff equation. 
\end{abstract}
%
%
%
\section{Introduction}
%
%
%
Kirchhoff equation was proposed by G. Kirchhoff in \cite{K} as the following 
integro-differential equation describing the vibration of an elastic string of length $l$: 
\begin{equation}\label{KC-1d}
  \pa_t^2 u(t,x) 
 -\(\ve^2 + \frac{1}{2l}\int^{l}_{0}\left|\pa_{x} u(t,x)\right|^2\,dx\)
  \pa_{x}^2 u(t,x) =0,
\end{equation}
where $x \in [0,l]$ and $\ve$ is a positive constant determined by the material of the string. 
The primitive equation \eqref{KC-1d} is generalized to the following wave equation with non-local nonlinearity: 
\begin{equation}\label{KC}
\pa_t^2 u(t,x) 
 -\Phi\(\sum_{j=1}^d \int_{\Om}\left| \pa_{x_j} u(t,x)\right|^2\,dx \)
  \sum_{j=1}^d \pa_{x_j}^2 u(t,x) =0,
\end{equation}
which is also called Kirchhoff equation, 
where $\Om$ is an open subset of $\R^d$, $x \in \Om$ and $\Phi$ satisfies 
\begin{equation}\label{Phi}
  \Phi \in C^1([0,\infty))
  \;\text{ and }\;
  \inf_{\eta \ge 0}\{\Phi(\eta)\} >0. 
\end{equation}
As with any other kind of wave equations, \eqref{KC} is studied with initial data 
\begin{equation}\label{KCIC}
u(0,x) = u_0(x),\;\; \pa_t u(0,x) = u_1(x),\;\; x\in \Om
\end{equation}
and some suitable boundary condition on $\pa\Om$ as an initial boundary value problem, or with \eqref{KCIC} for $\Om = \R^d$ as an initial value problem. 
Moreover, \eqref{KC} is naturally generalized as an abstract evolution equation in a real Hilbert space. 

Although Kirchhoff equation has been deeply investigated in many previous literature and an overview of the research and results refer to \cite{GG09} and \cite{S92}, 
here we only introduce a summary of the global solvability of \eqref{KC} and \eqref{KCIC}, which is the most fundamental and important problem. 
Briefly, the existence of a global solution is proved only one of the following cases: 
\begin{itemize}
\item[(i)] 
Analytic (or quasi-analytic) initial data 
(\cite{AS84, B40, N84, P75})

\item[(ii)] 
Small initial data (\cite{DS93, GH80})

\item[(iii)] 
Spectral-gap initial data (\cite{GG11, H06, M02})

\item[(iv)] 
Special nonlinearity $\Phi$ (\cite{P85})
\end{itemize}
Roughly speaking, the factors that made it possible to prove global solvability under these conditions are as follows: 
\begin{itemize}
\item 
The local solution could be extended to infinity by imposing a strong constraint on the high-frequency energy of the initial data: (i) and (iii). 
\item 
The non-linear problem could be reduced to the perturbation of the linear wave equation with constant coefficient: (ii). 
\item 
A special nonlinearity $\Phi$ (but different from $\Phi(\eta)=\ve^2+(2l)^{-1} \eta$) with the ``second order energy conservation'', which is crucial for the proof of global solvability: (iv). 
\end{itemize}
The necessity of these assumptions for global solvability should be a problem to be considered, but the other big problem for Kirchhoff equation is the non-existence of global solution, and it is completely unsolved. 
As mentioned above, many of the Kirchhoff equation remain unsolved, but the purpose of this paper is to prove the global solvability for semi-discrete Kirchhoff equation, which is the discretization of the initial value problem of \eqref{KC} with respect to spatial variable. 

%
%
%
\section{Discretization and main theorem}
%
%
%

Discretization of the partial differential equation \eqref{KC} can be considered for time and spatial variables, but here we consider discretization only for spatial0 variables; we shall call it semi-discretization. 

Let $l(\Z^d)$ be the set of all complex-valued infinite $d$ dimensional arrays. 
We define $l^2$ by 
\begin{equation*}
  l^2:=\left\{
  f = \{f[n]\}_{n \in \Z^d} \in l(\Z^d)\;;\;
  \|f\|:= \sqrt{\sum_{n\in\Z^d}|f[n]|^2}<\infty
  \right\}.
\end{equation*}
We note that $l^2$ is a Hilbert space with the inner product 
\[
  (f,g):=\sum_{n\in\Z^d}f[n] \ol{g[n]}
  \quad (f,g \in l^2).
\]

 For $f \in l(\Z^d)$ we define the forward and the backward difference operators $D_j^+$ and $D^-_j$ for $j=1,\ldots,d$ by 
\[
  D_j^+ f[n]:=f[n+e_j]-f[n]
  \;\text{ and }\;
  D_j^- f[n]:=f[n]-f[n-e_j]
  \quad
  (n\in \Z^d),
\]
where $\{e_1,\ldots,e_d\}$ is the standard basis of $\R^d$. 
Denoting $\nabla^{\pm} = (D_1^\pm,\ldots,D_d^\pm)$, the discrete Laplace operator $\Delta$ on $\Z^d$ is given by 
\[
  \Delta :=
  \nabla^+ \cd \nabla^- = D_1^+ D_1^- + \cdots + D_d^+ D_d^-, 
\]
that is, 
\[
  \Delta f[n]
 =\sum_{j=1}^d \(f[n+e_j] - 2f[n] + f[n-e_j]\)
  \quad
  (n\in \Z^d).
\]
\begin{lemma}\label{lemma-l2H1}
If $f,g \in l^2$, then 
$\|D_j^{\pm} f\| \le 2\|f\|$ 
and $(D_j^\pm f, g)=-( f, D_j^\mp g)$ 
for $j=1,\ldots,d$. 
\end{lemma}
\begin{proof}
Let $f,g \in l^2$. 
Then we have 
\begin{align*}
  \left\|D_j^{\pm} f\right\|^2
& =\sum_{n\in\Z^d} |f[n \pm e_j]-f[n]|^2
\\
&\le 2\(\sum_{n\in\N} |f[n \pm e_j]|^2 + \sum_{n\in\N}|f[n]|^2\)
 =4\|f\|^2. 
\end{align*}
Noting that $D_j^\pm f, D_j^\pm g \in l^2$, we have 
\begin{align*}
  \(D_j^\pm f, g\)
&=\pm\(\sum_{n\in\Z^d} f[n \pm e_j] \ \ol{g[n]}
   -\sum_{n\in\Z^d} f[n] \ \ol{g[n]}\)
\\
&=\pm\(\sum_{n\in\Z^d} f[n] \ \ol{g[n \mp e_j]}
   -\sum_{n\in\Z^d} f[n] \ \ol{g[n]}\)
\\
&=-\sum_{n\in\Z^d} f[n] \ \ol{D_j^\mp g[n]}
 =-\( f, D_j^\mp g\). 
\end{align*}
\end{proof}

Let us consider the following infinite system of initial value problem of second order ordinary differential equations: 
\begin{equation}\label{K}
\begin{cases}
  u''(t)[n] 
 -\Phi\(\left\|\nabla^+ u(t)\right\|^2\)
  \Delta u(t)[n] =0,
&(t,n) \in (0,\infty) \times \Z^d,
\\
  u(0)[n] = u_0[n],\quad
  u'(0)[n] = u_1[n],
& n \in \Z^d
\end{cases}  
\end{equation}
in which the domain $\Om = \R^d$ of \eqref{KC} and \eqref{KCIC} is discretized into the the lattice $\Z^d$, where 
$u'(t) = \frac{d}{dt}u(t) = \{\frac{d}{dt}u(t)[n]\}_{n\in\Z^d}$ 
and 
$\|\nabla^+ u(t)\|^2 = \sum_{j=1}^d \|D_j^+ u(t)\|^2$. 
We shall call the equation of \eqref{K} 
semi-discrete Kirchhoff equation. 

For the solution $u=u(t)$ of \eqref{K}, we define the energy functional $E(t;u)$ by 
\begin{equation*}
  E(t;u):=\frac12\( 
  \int^{\|\nabla^+ u(t)\|^2}_0 \Phi(\eta)\,d\eta
 +\left\|u'(t) \right\|^2
  \).
\end{equation*}
The following is the main theorem of this paper.
\begin{theorem}\label{Mth}
Let $\Phi$ satisfy \eqref{Phi}. 
If $u_0, u_1 \in l^2$, 
then the unique global solution $u(t) \in C^2([0,\infty);l^2)$ 
of \eqref{K} exists. 
Moreover, the following energy conservation is established: 
\begin{equation}\label{E-conserv}
  E(t;u) \equiv E(0;u) \quad (t>0).
\end{equation}
\end{theorem}

\begin{remark}
If $d=1$, $\Om = [1,N]$ for $N \in \N$ and $N \ge 2$, then a corresponding semi-discrete problem of \eqref{KC} and \eqref{KCIC} with Dirichlet boundary condition $u(t,1)=u(t,N)=0$ is the following $N$ system of initial value problem of second order ordinary differential equations: 
\begin{equation*}
\begin{cases}
  u''(t)[n] 
 -\Phi\(\sum_{m=1}^{N} |D_1^+ u(t)[m]|^2\)
  \(u(t)[n+1]-2u(t)[n]+u(t)[n-1]\) =0, 
\\
  u(0)[n] = u_0[n],\quad
  u'(0)[n] = u_1[n],
\\
  u(t)[m]=u(t)[N+m-1]=0 \;\; (m=0,1,2)
\end{cases}  
\end{equation*}
for $t > 0$ and $n=1,\ldots,N$. 
Then the existence of a unique global solution $u(t)=\{u(t)[n]\}_{n=1}^N$ is trivial. 
\end{remark}

\begin{remark}
Time global energy estimates for linear semi-discrete wave equations are studied in \cite{H21}, but no results are known for global existence of a solution of semi-discrete problem \eqref{K} as far as the author knows. 
\end{remark}

Without loss of generality, we assume that 
\[
  \inf_{\eta \ge 0}\{\Phi(\eta)\} = 1
\]
in following sections.

%
%
%
\section{Linear semi-discrete wave equation}
%


Let $\phi \in C^1([0,\infty))$ satisfy $\inf_{t \ge 0}\{\phi(t)\} \ge 1$. 
For $w(t) \in C^2([0,\infty);l^2)$, 
we define $\tE_0(t;w,\phi)$ and $\tE_1(t;w,\phi)$ by 
\begin{equation*}
  \tE_0(t;w,\phi):=\frac{1}{2}\(
  \phi(t)\left\| \nabla^+ w(t)\right\|^2 + \left\|w'(t) \right\|^2
  \) 
\end{equation*}
and
\begin{equation*}
  \tE_1(t;w,\phi):=\frac{1}{2}\(
  \phi(t)\left\|\Delta w(t)\right\|^2 + \left\|\nabla^+ w'(t) \right\|^2
  \). 
\end{equation*}
We consider the following initial value problem of linear semi-discrete wave equation with time dependent coefficient: 
\begin{equation}\label{L}
\begin{cases}
  w''(t)[n]-\phi(t)\Delta w(t)[n]=0,
  & (t,n) \in (0,\infty)\times \Z^d,
\\
  w(0)[n] = u_0[n],\quad
  w'(0)[n] = u_1[n],
& n \in \Z^d. 
\end{cases} 
\end{equation}

\begin{proposition}\label{prop-L}
If $u_0,u_1 \in l^2$, then the unique solution 
$w(t) \in C^2([0,\infty);l^2)$ of \eqref{L} exists. 
Moreover, the following estimate is established$:$ 
\begin{equation}\label{eest-L}
  \tE_m(t;w,\phi) \le  \exp\(\Psi_1(t;\phi)t\) \tE_m(0;w,\phi) 
  \quad
  (m=0,1)
\end{equation}
for any $t>0$, where
\begin{equation*}
  \Psi_1(t;\phi):=\max_{0\le s \le t}\{|\phi'(s)|\}. 
\end{equation*}
\end{proposition}
\begin{proof}
For $\th = (\th_1,\ldots,\th_d) \in \T^d$ and $\T:=[-\pi,\pi]$, we define 
\[
  \xi(\th):=\(2\sin\frac{\th_1}{2},\ldots,2\sin\frac{\th_d}{2}\), 
\]
and consider the following initial value problme of second order ordinary differential equation with parameter $\th$: 
\begin{equation*}
\begin{cases}
  v''(t;\th) + \phi(t) |\xi(\th)|^2 v(t;\th) = 0, \quad t > 0,
\\
  v(0;\th) = v_0(\th),\quad 
  v'(0;\th) = v_1(\th),
\end{cases} 
\end{equation*}
that has a unique global solution. 
Setting the initial data by 
\begin{equation*}
  v_0(\th):=\sum_{n\in\Z^d} e^{-i \th \cd n} u_0[n]
  \;\;\text{ and }\;\;
  v_1(\th):=\sum_{n\in\Z^d} e^{-i \th \cd n} u_1[n], 
\end{equation*}
where $\th\cd n = \th_1 n_1 + \cdots + \th_d n_d$, 
by Parseval's identity, we have  
\begin{align*}
  \int_{\T^d} |v_k(\th)|^2\,d\th
&=\int_{\T^d} \(\sum_{n\in\Z^d} e^{-i \th \cd n} u_k[n]\) 
              \ol{\(\sum_{m\in\Z^d} e^{-i \th \cd m} u_k[m]\)} \,d\th
\\
&=\sum_{n\in\Z^d}\sum_{m\in\Z^d} 
  u_k[n] \ \ol{u_k[m]} \ 
  \prod_{j=1}^d \int_{\T}  e^{i \th_j (m_j-n_j)} \,d\th_j 
\\
&=(2\pi)^d \|u_k\|^2
\end{align*}
for $k=0,1$, it follows that $v_0, v_1 \in L^2(\T^d)$. 
Moreover, setting $w(t)$ by 
\begin{equation*}
  w(t)[n]:=\frac{1}{(2\pi)^d} \int_{\T^d} e^{i n\cd \th} v(t;\th)\,d\th
\end{equation*}
for $n\in\Z^d$, we see that $w(0)=u_0$, $w'(0)=u_1$ and 
\begin{align*}
  w''(t)[n]
&=\frac{1}{(2\pi)^d} \int_{\T^d} e^{i n\cd \th} v''(t;\th)\,d\th
 =-\phi(t)\frac{1}{(2\pi)^d} \int_{\T^d} 
    e^{i n\cd \th} |\xi(\th)|^2 v(t;\th)\,d\th
\\
&=\phi(t)\frac{1}{(2\pi)^d}
  \int_{\T^d}  \sum_{j=1}^d
    \(e^{i(n+e_j)\cd\th}-2e^{in\cd\th} + e^{i(n-e_j)\cd\th}\) v(t;\th)\,d\th
\\
&=\phi(t)\sum_{j=1}^d\(w(t)[n+e_j] - 2w(t)[n] + w(t)[n-e_j]\)
\\
&=\phi(t)\Delta w(t)[n],
\end{align*}
it follows that $w(t)$ is the solution of \eqref{L}. 
Noting the equalities
$D_j^+ D_j^- = D_j^- D_j^+$ ($j=1,\ldots,d$), 
\begin{align*}
  \frac{d}{dt}\frac12\left\| \nabla^+ w(t)\right\|^2
&=\sum_{j=1}^d \Re\(D_j^+ w'(t), D_j^+ w(t)\)
\\
&= -\sum_{j=1}^d \Re\(w'(t), D_j^- D_j^+ w(t)\)
\\
&= -\Re\(w'(t), \Delta w(t)\)
\end{align*}
and
\begin{align*}
  \frac{d}{dt}\frac12\left\| \nabla^+ w'(t)\right\|^2
 = -\Re\(w''(t), \Delta w'(t)\), 
\end{align*}
we have 
\begin{align*}
  \frac{d}{dt}\tE_0(t;w,\phi)
&=\frac12\phi'(t)\left\| \nabla^+ w(t)\right\|^2
  \le |\phi'(t)| \tE_0(t;w,\phi) 
\end{align*}
and
\begin{align*}
  \frac{d}{dt}\tE_1(t;w,\phi)
&=\frac12\phi'(t)\left\|\Delta w(t)\right\|^2 
 +\phi(t)\Re\(\Delta w'(t), \Delta w(t)\)
 -\Re\(w''(t), \Delta w'(t)\)
\\
&=\frac12\phi'(t)\left\| \Delta w(t)\right\|^2
  \le |\phi'(t)| \tE_1(t;w,\phi). 
\end{align*}
It follows that 
\begin{align*}
  \tE_m(t;w,\phi) 
  \le \exp\(\int^t_0 |\phi'(s)|\,ds\) \tE_m(0;w,\phi) 
  \le \exp\(\Psi_1(t;\phi) t \) \tE_m(0;w,\phi)
\end{align*}
for $m=0,1$ by Gronwall's inequality. 
The uniqueness of the solution is immediately follow from \eqref{eest-L}. 
\end{proof}

%
%
%
\section{Local solvability for semi-discrete Kirchhoff equation}
%
%
%

Let us consider the following initial value problem of semi-discrete Kirchhoff equation with an initial time $T_0 \ge 0$: 
\begin{equation}\label{KT0}
\begin{cases}
  u''(t)[n] 
 -\Phi\(\left\|\nabla^+ u(t)\right\|^2\)
  \Delta u(t)[n] =0,
&(t,n) \in (T_0,\infty) \times \Z^d,
\\
  u(T_0)[n] = \tilde{u}_0[n],\quad
  u'(T_0)[n] = \tilde{u}_1[n],
& n \in \Z^d. 
\end{cases}  
\end{equation}

For $L>0$ and $\tilde{u}_0, \tilde{u}_1 \in l^2$, we define
\[
  \Phi_0(L):=\max_{0 \le \eta \le L}\{|\Phi(\eta)|\},
  \quad
  \Phi_1(L):=\max_{0 \le \eta \le L}\{|\Phi'(\eta)|\},
\]
\begin{equation*}
 L_0=L_0\(\tilde{u}_0,\tilde{u}_1\):=
  \Phi\(\|\nabla^+ \tilde{u}_0\|^2\) \left\| \nabla^+ \tilde{u}_0\right\|^2 
  + \| \tilde{u}_1\|^2
\end{equation*}
and
\begin{equation*}
 L_1=L_1\(\tilde{u}_0,\tilde{u}_1\):=
  \Phi\(\|\nabla^+ \tilde{u}_0\|^2\) \left\| \Delta \tilde{u}_0 \right\|^2 
  + \left\| \nabla^+ \tilde{u}_1 \right\|^2. 
\end{equation*}
\begin{proposition}\label{Prop_ls}
If $\tilde{u}_0, \tilde{u}_1 \in l^2$, then there exists a unique solution $u(t)$ of \eqref{KT0} on $[T_0,T_0 + M_1^{-1}]$ with 
\begin{equation*}
  M_1:=e(L_0+L_1) \Phi_1(e L_0)
\end{equation*}
such that $u(t) \in C^2([T_0,T_0 + M_1^{-1}];l^2)$. 
\end{proposition}

From now on we suppose that $T_0=0$ without loss of generality. 
In order to prove Proposition \ref{Prop_ls}, we consider the following series of linear problems: 
\begin{equation}\label{L_nu}
\begin{cases}
  w_\nu''(t)[n]
  -\Phi\(\|\nabla^+ w_{\nu-1}(t)\|^2\) \Delta w_{\nu}(t)[n]=0,
   & (t,n) \in (0,\infty) \times \Z^d,\\[2mm]
  w_\nu(0)[n]=u_0[n],\quad w_\nu'(0)[n]=u_1[n], & n \in \Z^d
\end{cases} 
\end{equation}
for $\nu=1,2,\cdots$, and $w_0(t)=\tilde{u}_0$. 
\begin{lemma}\label{lemm-ls1}
The following estimates are established$:$
\begin{equation}\label{eq0_lemm-ls1}
  \tE_m(t;w_\nu,\phi_{\nu-1})
  \le \frac{eL_m}{2} \quad(m=0,1)
\end{equation}
for any $\nu \ge 1$ and $t \le M_1^{-1}$. 
It follows that 
\begin{equation}\label{eq1_lemm-ls1}
  \max_{0\le t \le T}
  \left\{\Phi\(\|\nabla^+ w_{\nu}(t)\|^2\)\right\} 
  \le \Phi_0\(e L_0\)
\end{equation}
and
\begin{equation}\label{eq2_lemm-ls1}
  \max_{0\le t \le T}
  \left\{\left| \frac{d}{dt}\Phi\(\|\nabla^+ w_{\nu}(t)\|^2\)\right|
  \right\}
  \le M_1.
\end{equation}
\end{lemma}
\begin{proof}
We denote 
\begin{equation*}
  \phi_\nu(t):=
  \begin{cases}
    \Phi\(\|\nabla^+ u_0\|^2\) & \text{\it for $\nu = 0$}, \\
    \Phi\(\|\nabla^+ w_{\nu}(t)\|^2\) & \text{\it for $\nu \ge 1$}.
  \end{cases}
\end{equation*}
If \eqref{eq0_lemm-ls1} holds, then we have 
\begin{equation*}
  \Phi\(\|\nabla^+ w_{\nu}(t)\|^2\)
  \le \Phi\(2\tE_0(t;w_\nu,\phi_{\nu-1})\)
  \le \Phi_0\(e L_0\)
\end{equation*}
and 
\begin{align*}
  \left|\frac{d}{dt}\Phi\(\|\nabla^+ w_\nu(t)\|^2\)\right|
&\le 2\sum_{j=1}^d \left| \Re\(D_j^+ w_\nu'(t),D_j^+ w_\nu(t)\) \right|
  \left|\Phi'\(\|\nabla^+ w_\nu(t)\|^2\)\right|
\\
&\le \( \|\nabla^+ w_\nu'(t)\|^2 + \|\nabla^+ w_\nu(t)\|^2 \)
  \Phi_1\(2\tE_0(t;w_\nu,\phi_{\nu-1})\)
\\
&\le \( 2\tE_1(t;w_\nu,\phi_{\nu-1}) + 2\tE_0(t;w_\nu,\phi_{\nu-1}) \)
  \Phi_1(eL_0)
\\
&\le e(L_1 + L_0) \Phi_1(eL_0), 
\end{align*}
that is,
\[
  \Psi_1\(t;\phi_\nu\) \le M_1. 
\]
Thus \eqref{eq1_lemm-ls1} and \eqref{eq2_lemm-ls1} hold. 
By applying Proposition \ref{prop-L} with 
$\phi(t)=\phi_0$ and noting that $\Psi_1(t;\phi_0) = 0$, 
we have 
\begin{equation}\label{est_tEm1}
  \tE_m(t;w_1,\phi_0)
  \le \tE_m(0;w_1,\phi_0)
  = \frac{L_m}{2} \quad(m=0,1)
\end{equation}
for any $t \ge 0$. 
Suppose that \eqref{eq0_lemm-ls1} is valid for $\nu = \mu$. 
Noting that 
$2\tE_m(0;w_{\mu+1},\phi_\nu)= L_m$, 
by Proposition \ref{prop-L} for \eqref{L_nu} with $\nu=\mu+1$ and \eqref{eq2_lemm-ls1}, we have 
\begin{align*}
 \tE_0(t;w_{\mu+1},\phi_\mu)
 \le \exp\(\Psi_1\(t;\phi_\mu\) t\)  \tE_0(0;w_{\mu+1},\phi_\mu) 
  \le \frac{\exp\(M_1 t\) L_0}{2}
  \le \frac{e L_0}{2}
\end{align*}
and
\begin{align*}
  \tE_1(t;w_{\mu+1},\phi_\mu)
  \le
  \frac{\exp\(M_1 t\) L_1}{2}
  \le \frac{e L_1}{2}
\end{align*}
for $t \le M_1^{-1}$. 
Therefore, \eqref{eq0_lemm-ls1} with $\nu=\mu+1$ is established for 
$t \le M_1^{-1}$, and thus \eqref{eq0_lemm-ls1} is valid for any $\nu \in \N$. 
\end{proof}

\noindent
\textit{Proof of Proposition \ref{Prop_ls}}.\; 
For $\nu=1,2,\ldots$ we set 
\[
  y_\nu(t):=w_{\nu+1}(t) - w_\nu(t).  
\]
Then $y_\nu(t)$ is the solution of the following problem: 
\begin{equation}\label{y_nu}
\begin{cases}
  y_\nu''(t)-\phi_\nu(t) \Delta y_\nu(t)
  =\(\phi_\nu(t)-\phi_{\nu-1}(t)\) \Delta w_{\nu}(t),
   & (t,n) \in (0,T] \times \Z^d,\\[2mm]
  y_\nu(0)[n]=y_\nu'(0)[n]=0, & n \in \Z^d,
\end{cases} 
\end{equation}
where $T=M_1^{-1}$. 
For the solutions of \eqref{L_nu} and \eqref{y_nu}, we define  
\[
  H(t;y_\nu,\phi_\nu):=
  \frac12\( \phi_\nu(t)\|\nabla^+  y_\nu(t)\|^2 + \|y_\nu'(t)\|^2\).
\]
By mean value theorem, Lemma \ref{lemm-ls1} and \eqref{est_tEm1},
there exists $\ka\in \R$ satisfying $0<\ka<1$ such that 
\begin{align*}
  \left|\phi_1(t)-\phi_0(t)\right| 
  &\le 
  \left| \|\nabla^+ w_1(t)\|^2 - \|\nabla^+ w_0(t)\|^2 \right|
\\
  &\qquad
  \left|\Phi'\(\ka\|\nabla^+ w_1(t)\|^2+(1-\ka)\|\nabla^+ w_0(t)\|^2\)\right|
\\
  &\le L_0
  \left|\Phi'\(2\ka \tE_0(t;w_1,\phi_0) +2(1-\ka)\tE_0(0;w_0,\phi_0)\)\right|
\\
  &\le
  L_0 \Phi_1(L_0)
\end{align*}
for any $t \ge 0$, and 
\begin{align*}
  \left|\phi_\nu(t)-\phi_{\nu-1}(t)\right| 
  &\le 
  \left|\Phi'\(\ka\|\nabla^+ w_{\nu}(t)\|^2+(1-\ka)\|\nabla^+ w_{\nu-1}(t)\|^2\)\right| 
\\
  &\qquad
  \left| \|\nabla^+ w_{\nu}(t)\|^2 - \|\nabla^+ w_{\nu-1}(t)\|^2 \right|
\\
  &\le
  \Phi_1(eL_0)
  \(\|\nabla^+ w_{\nu}(t)\| + \|\nabla^+ w_{\nu-1}(t)\| \)
\\
  &\qquad
  \left| \|\nabla^+ w_{\nu}(t)\| - \|\nabla^+ w_{\nu-1}(t)\| \right|
\\
  &\le
  2 \Phi_1(eL_0) \sqrt{eL_0}
  \ \|\nabla^+ (w_{\nu}(t) - w_{\nu-1}(t))\|
\\
  &\le 
  2\Phi_1(eL_0) \sqrt{eL_0}
  \sqrt{2H(t;y_{\nu-1},\phi_{\nu-1})} 
\end{align*}
for $\nu = 2,3,\ldots$ and $t \le T$. 
Therefore, by \eqref{eq0_lemm-ls1}, \eqref{eq2_lemm-ls1} and \eqref{est_tEm1}, 
we have 
\begin{align*}
  \frac{d}{dt}H(t;y_1,\phi_1)
&= \frac12\phi_1'(t) \|\nabla^+ y_1(t) \|^2
  + \(\phi_1(t) - \phi_0(t)\) \Re\( \Delta w_1(t),y_1'(t)\)
\\
&\le 
  M_1 H(t;y_1,\phi_1)
 + L_0 \Phi_1(L_0) \|\Delta w_1(t)\| \ \|y_1'(t)\|
\\
&\le 
  M_1 H(t;y_1,\phi_1)
 + L_0  \Phi_1(L_0) \sqrt{L_1} \sqrt{2H(t;y_1,\phi_1)}
\end{align*}
and
\begin{align*}
  \frac{d}{dt}H(t;y_\nu,\phi_\nu)
&\le 
  M_1 H(t;y_\nu,\phi_\nu)
\\
&\quad
 +2\Phi_1(eL_0) \sqrt{eL_0} \sqrt{2H(t;y_{\nu-1},\phi_{\nu-1})} 
  \|\Delta w_{\nu}(t)\| \ \|y_\nu'(t)\| 
\\
&\le 
  M_1 H(t;y_\nu,\phi_\nu)
\\
&\quad
 +2e \sqrt{L_0L_1} \Phi_1(eL_0) \sqrt{2H(t;y_{\nu-1},\phi_{\nu-1})} 
  \sqrt{2H(t;y_\nu,\phi_\nu)}. 
\end{align*}
Setting $y_0$ to satisfy $H(t;y_0,\phi_0) \equiv 1$, we have 
\begin{align*}
  \frac{d}{dt}\sqrt{H(t;y_\nu,\phi_\nu)}
\le 
  \frac{M_1}{2} \sqrt{H(t;y_\nu,\phi_\nu)}
 +L \sqrt{H(t;y_{\nu-1},\phi_{\nu-1})} 
\end{align*}
for any $\nu \ge 1$, where 
\[
  L=\frac12\max\left\{
    L_0  \sqrt{2L_1} \Phi_1(L_0), 
    4e\sqrt{L_0L_1} \Phi_1(eL_0) 
  \right\}. 
\]
By Gronwall's inequality, Stirling's formula and noting that $H(0;y_\nu,\phi_\nu)=0$ for any $\nu \ge 1$, we have 
\begin{align*}
\sqrt{H(t;y_\nu,\phi_\nu)} 
&\le L e^{\frac{M_1}{2} t} \int^t_0 e^{-\frac{M_1}{2} s_1}
  \sqrt{H(s_1;y_{\nu-1},\phi_{\nu-1})} \,ds_1
\\
&\le L^2 e^{\frac{M_1}{2} t} \int^t_0 \int^{s_1}_0 
  e^{-\frac{M_1}{2} s_2} \sqrt{H(s_2;y_{\nu-2},\phi_{\nu-2})} \,ds_2 \,ds_1
\\
&\;\; \vdots
\\
&\le L^{\nu} e^{\frac{M_1}{2} t} \int^t_0 \int^{s_1}_0  
  \cdots \int^{s_{\nu-1}}_0  e^{-\frac{M_1}{2} s_{\nu}} 
  \sqrt{H(s_{\nu};y_0,\phi_0)}\,ds_{\nu} \cdots ds_1
\\
&\le  L^{\nu} e^{\frac{M_1}{2} t} \int^t_0 \int^{s_1}_0  
  \cdots \int^{s_{\nu-1}}_0   \,ds_{\nu} \cdots ds_1
 =e^{\frac{M_1}{2} t} \frac{(Lt)^\nu}{\nu!}
\\
&\le e^{\frac{M_1}{2} T} \frac{(LT)^\nu}{\nu!}
 \le \frac{e^{\frac{M_1}{2}T}}{\sqrt{2\pi \nu}} \(\frac{eLT}{\nu}\)^\nu.
\end{align*}
Therefore, for $\nu,\mu \in \N$ satisfying $\nu > \mu$ 
and $\mu \ge \max\{2eLT,e^{M_1 T}/\pi\}$, we have 
\begin{align*}
\|w_\nu'(t)-w_\mu'(t)\| 
& \le
  \sum_{j=\mu}^{\nu-1} \|w_{j+1}'(t)-w_{j}'(t)\|
\\
&\le \sum_{j=\mu}^{\nu-1} \sqrt{2H(t;y_j,\phi_j)}
 \le \sum_{j=\mu}^{\nu-1} \(\frac12\)^j
\\
& \le \(\frac12\)^{\mu-1}
  \to 0 \;\;(\mu\to\infty), 
\end{align*}
it follows that 
$\{w_\nu'(t)\}_{\nu=1}^\infty$ is a Cauchy sequences in $l^2$ for any $t\in[0,T]$, 
and $\{w_\nu'(t)[n]\}_{\nu=1}^\infty$ 
is a uniformly convergent sequence on $[0,T]$ for any $n\in\Z^d$. 
Consequently, there exist $W(t) \in C^0([0,T];l^2)$ such that 
\begin{equation}\label{w_nu'-W}
  \lim_{\nu\to\infty} 
  \|w_\nu'(t)-W(t)\|
   =0. 
\end{equation}
Setting $u(t) \in C^1([0,T];l^2)$ as 
\[
  u(t)[n]:=\int^t_0 W(s)[n]\,ds + u_0[n]
  \;\;\text{ for }\;\;  n \in \Z^d,
\]
we have 
\begin{align*}
  |w_\nu(t)[n]-u(t)[n]|
&=\left|\int^t_0 w_\nu'(s)[n]\,ds - \int^t_0 W(s)[n]\,ds\right|
\\
&\le \int^t_0 \left| w_\nu'(s)[n] -  W(s)[n] \right|\,ds
\\
&\le T \sup_{s\in[0,T]} \left\{\left| w_\nu'(s)[n] -  W(s)[n] \right|\right\},
\end{align*}
and thus $\lim_{\nu\to\infty} \|w_\nu(t) - u(t)\| =0$ by \eqref{w_nu'-W}. 
Moreover, by Lemma \ref{lemma-l2H1}, we have 
\begin{equation*}
  \lim_{\nu\to\infty}\|\nabla^+ w_\nu(t) - \nabla^+ u(t)\|
 \le 2\sqrt{d}\lim_{\nu\to\infty}\|w_\nu(t) - u(t)\|
 =0
\end{equation*}
and
\begin{align*}
  \lim_{\nu\to\infty}\|\Delta w_\nu(t) - \Delta u(t)\|
 \le 4d\lim_{\nu\to\infty}\|w_\nu(t) - u(t)\|
 =0, 
\end{align*}
it follows that $u(t)$ is a solution of \eqref{K} on $[0,T]$. 
Noting that $\Phi\(\|\nabla^+ u(t)\|^2\) \in C^1([0,T];l^2)$, 
we have $u''(t) \in C^0([0,T];l^2)$, thus we conclude the proof of Proposition \ref{Prop_ls}. 
\qed

%
\section{Global solvability}
%

The energy conservation is a characteristic property for \eqref{KC}, 
and semi-discrete problem \eqref{K} also has the corresponding property as far as the solution exists. 
%
\begin{lemma}[Energy conservation]\label{lemm-EC}
If $T>0$ and the solution $u(t) \in C^2([0,T];l^2)$ of \eqref{K} exists,
then the energy conservation \eqref{E-conserv} is established for any $t\in [0,T]$. 
\end{lemma}
\begin{proof}
Since $D_j^\pm D_j^\mp u(t), D_j^+ u'(t) \in l^2$ for $j=1,\ldots,d$ by Lemma \ref{lemma-l2H1}, we have 
\begin{align*}
  \frac{d}{dt}E(t;u)
&=\Phi\(\|\nabla^+ u(t)\|^2\) \sum_{j=1}^d \Re\(D_j^+ u'(t),D_j^+ u(t)\) 
 +\Re\(u'(t),u''(t)\)
\\
&=-\Phi\(\|\nabla^+ u(t)\|^2\)\sum_{j=1}^d \Re\(u'(t),D_j^- D_j^+ u(t)\)
 +\Re\(u'(t),u''(t)\)
\\
&=-\Phi\(\|\nabla^+ u(t)\|^2\)\Re\(u'(t), \Delta u(t)\)
 +\Re\(u'(t),u''(t)\)
\\
&=0.
\end{align*}
It follows that  $E(t;u) = E(0;u)$ for any $t\in[0,T]$. 
\end{proof}

\noindent
\textit{Proof of Theorem \ref{Mth}.}\;
Let $u_0,u_1 \in l^2$ and $u(t)$ be the solution of \eqref{K} on $[0,M_1^{-1}]$, where $M_1$ is given in Proposition \ref{Prop_ls}. 
We note that the following estimates hold:
\begin{equation*}\label{lemma_Em-e1}
  \|\nabla^+ u(t)\|^2+\|u'(t)\|^2
  \le 2E(t;u) = 2E(0;u),
\end{equation*}
\begin{equation}\label{L0E}
  L_0(u(t),u'(t))
 \le 2\Phi_0\(2E(t;u)\)E(t;u) 
 = 2\Phi_0\(2E(0;u)\)E(0;u)
 =:E_0
\end{equation}
and
\begin{equation}\label{L1L0}
  L_1(u(t),u'(t)) \le 4d L_0(u(t),u'(t)) 
\end{equation}
for any $t \in [0,M_1^{-1}]$ by Lemma \ref{lemma-l2H1} and Lemma \ref{lemm-EC}. Setting 
\[
  \de_1:=\frac{1}{e(1+4d)E_0 \Phi_1(e E_0)} 
\]
and applying Proposition \ref{Prop_ls} with 
\[
  T_0=\de_1  
  \le \frac{1}{e(1+4d)L_0(u_0,u_1) \Phi_1(e L_0(u_0,u_1))}
  \le M_1^{-1},
\]
$\tilde{u}_0 = u(\de_1)$ and $\tilde{u}_1 = u'(\de_1)$, 
the existence time of the solution is extended 
from $[0, M_1^{-1}]$ to 
$[0,\de_1 + M_2^{-1}]$, 
where 
\[
  M_2:= e(L_0(u(\de_1),u'(\de_1))+L_1(u(\de_1),u'(\de_1)))
    \Phi_1(eL_0(u(\de_1),u'(\de_1))). 
\]
Moreover, by \eqref{L0E} and \eqref{L1L0}, we have 
\begin{align*}
  M_2 
 &\le e(1+4d) L_0(u(\de_1),u'(\de_1))\Phi_1(eL_0(u(\de_1),u'(\de_1)))\\
 &\le e(1+4d) E_0 \Phi_1(eE_0) = \de_1^{-1}. 
\end{align*}
Thus the existence time of the solution is ensured at least on $[0, 2\de_1]$. 
Analogously, by applying Proposition \ref{Prop_ls} with 
$T_0 = n\de_1$, $\tilde{u}_0=u(n\de_1)$ and $\tilde{u}_1=u'(n\de_1)$, 
the existence time of the solution of \eqref{K} can be extend on 
$[0,(n+1)\de_1]$, 
and also the energy conservation is established on $[0,(n+1)\de_1]$ 
for $n=2,3,\ldots$. 
Thus the proof of Theorem \ref{Mth} is concluded. 
%
\qed

%
\section{Concluding remarks}
%

Although this paper considers the most basic semi-discrete model of Kirchhoff equations, the following problems should be studied as the next step. 

\begin{itemize}
\item 
The equation \eqref{K} is a discretized model on a fixed lattice $\Z^d$, thus the result of the global solvability is in some sense isolated from the results of continuous models. 
A unified understanding of continuous and semi-discrete models should be an interesting and important problem, and considering some asymptotic analysis with the interval of the lattice approaches zero is considered to be one method. 

\item 
According to \cite{GG10, H02}, the assumption \eqref{Phi} is crucial for the continuous model \eqref{KC} even for the local solvability, and it comes from the linear problem in \cite{CDS79}. 
On the other hand, $\phi \in C^1([0,\infty))$ for linear problem \eqref{L} is not necessarily required in \cite{H21} for semi-discrete model, thus $\Phi \in C^1([0,\infty))$ in \eqref{Phi} may also not be necessary. 
Furthermore, it is an interesting question whether the positivity of $\Phi$ and $\phi$ are required both non-linear and linear problems. 
\end{itemize}

\end{document}